\documentclass[a4paper,11pt]{amsart}
\usepackage{amsmath,amssymb,amsthm,enumerate,graphicx,mathtools}
\usepackage[margin=2.5cm]{geometry}

\usepackage{xypic} 
\usepackage{mathrsfs} 
\usepackage{xspace}

\usepackage{hyperref}

\usepackage[usenames, dvipsnames]{color}
\newtheorem{thm}{Theorem}[section]

\newtheorem{prop}[thm]{Proposition}
\newtheorem{lem}[thm]{Lemma}
\newtheorem{cor}[thm]{Corollary}

\theoremstyle{definition}

\newtheorem*{claim*}{Claim}

\newtheorem{example}[thm]{Example}

\newcommand{\ms}[1]{\mathscr{#1}} 

\newcommand{\Zb}{\mathbb{Z}}
\newcommand{\Nb}{\mathbb{N}}
\newcommand{\Rb}{\mathbb{R}}
\newcommand{\Cb}{\mathbb{C}}
\newcommand{\Qb}{\mathbb{Q}}
\newcommand{\Tb}{\mathbb{T}}
\newcommand{\Kb}{\mathbb{K}}
\newcommand{\Pb}{\mathbb{P}}
\newcommand{\Fb}{\mathbb{F}}

\newcommand{\Wr}{\; \mathrm{Wr} \;}

\newcommand{\lcsc}{l.c.s.c.\@\xspace}
\newcommand{\defbold}{\textbf}

\newcommand{\inv}{^{-1}}
\newcommand{\triv}{\{1\}}

\newcommand{\divi}{\mathrm{div}}

\newcommand{\Mon}{\mathrm{Mon}}

\newcommand{\CC}{\mathrm{C}}
\newcommand{\N}{\mathrm{N}}

\newcommand{\Aut}{\mathrm{Aut}}

\newcommand{\cgrp}[1]{\overline{\langle #1 \rangle}}

\newcommand{\grp}[1]{\langle #1 \rangle}
\newcommand{\ol}[1]{\overline{#1}}

\begin{document}

\title[Abelian minimal closed normal subgroups of l.c.s.c. groups]{A classification of the abelian minimal closed normal subgroups of locally compact second-countable groups}
\author{Colin D. Reid}
\address{  University of Newcastle,
   School of Mathematical and Physical Sciences,
   University Drive,
   Callaghan NSW 2308, Australia}
\email{colin@reidit.net}
\thanks{The author is a Postdoctoral Research Associate funded through ARC project \emph{Zero-dimensional symmetry and its ramifications} (FL170100032).}

\begin{abstract}
We classify the locally compact second-countable (\lcsc) groups $A$ that are abelian and topologically characteristically simple.  All such groups $A$ occur as the monolith of some soluble \lcsc group $G$ of derived length at most $3$; with known exceptions (specifically, when $A$ is $\mathbb{Q}^n$ or its dual for some $n \in \mathbb{N}$), we can take $G$ to be compactly generated.  This amounts to a classification of the possible isomorphism types of abelian chief factors of \lcsc groups, which is of particular interest for the theory of compactly generated locally compact groups.
\end{abstract}

\maketitle

\section{Introduction}

Let $G$ be a locally compact second-countable (\lcsc) group.  A \defbold{minimal closed normal subgroup} of $G$ is a nontrivial closed normal subgroup $N$ of $G$ such that, whenever $K$ is a closed normal subgroup of $G$ contained in $N$, then $K = N$ or $K = \triv$.  More specifically, we define the \defbold{(topological) monolith} $\Mon(G)$ of $G$ to be the intersection of all nontrivial closed normal subgroups of $G$.

The main goal of the article is to prove the following classification theorem.

\begin{thm}\label{mainthm}\
\begin{enumerate}[(i)]
\item Let $n \in \Nb$ and let $p$ be a prime number.  The following abelian groups occur as the monolith of some compactly generated soluble \lcsc group of derived length at most $3$:
\[
C^n_p; \; C^{\aleph_0}_p; \; C^{(\aleph_0)}_p; C^{\aleph_0}_p \oplus C^{(\aleph_0)}_p;  \; \Rb^n; \; \Qb^{(\aleph_0)}; \; \widehat{\Qb}^{\aleph_0}; \; \Qb^n_p; \; \Qb_p(\aleph_0).
\]
\item The groups $\Qb^n \text{ and } \widehat{\Qb}^n \; (n \in \Nb)$ cannot occur as minimal closed normal subgroups in any compactly generated \lcsc group, but do occur as the monolith of some soluble \lcsc group of derived length at most $3$.
\item Every abelian \lcsc group that is topologically characteristically simple is isomorphic to one of the groups listed in (i) or (ii).
\end{enumerate}
\end{thm}

Here $A^\kappa$ denotes the direct product of $\kappa$ copies of $A$; $A^{(\kappa)}$ the direct sum of $\kappa$ copies of $A$; and $\Qb_p(\kappa)$ is the local direct product of $\kappa$ copies of $\Qb_p$ over $\Zb_p$, that is, the group of functions from a set of size $\kappa$ to $\Qb_p$ with all but finitely many values in $\Zb_p$.

The topologically characteristically simple abelian groups in Theorem~\ref{mainthm} are listed in such a way that there is no redundancy, that is, no two groups listed are isomorphic unless given in the same form with the same parameters.  The group $C^{\aleph_0}_p \oplus C^{(\aleph_0)}_p$ is also known as the additive group of the field $\Fb_p((t))$ of formal Laurent series over the field of $p$ elements.  The groups naturally fall into five disjoint families, namely
\[
C^{\kappa}_p \oplus C^{(\kappa')}_p; \quad \Rb^n; \quad \Qb^{(\kappa)}; \quad \widehat{\Qb}^{\kappa}; \quad \Qb_p(\kappa);
\]
where $n \in \Nb$, $p$ is a prime number, $\kappa \in \Nb \cup \{\aleph_0\}$ and $\kappa' \in \{0,\aleph_0\}$.  For each family we will give a characterization of that family in the class of abelian \lcsc groups, in terms of basic properties of \lcsc groups.  These characterizations do not directly invoke the property of being topologically characteristically simple.

\subsection{Context}

All of the abelian groups listed in Theorem~\ref{mainthm} are well-known to scholars of locally compact abelian groups and are discussed in standard references on the subject, such as \cite{Armacost} and \cite{HR}.  The \defbold{homogeneous} locally compact abelian groups $A$, that is, such that $\Aut(A)$ acts transitively on the nonzero elements of $A$, were classified by L. Robertson (\cite[Theorem~4.13]{Robertson}): in the second-countable case these are
\[
C^{\kappa}_p \oplus C^{(\kappa')}_p; \quad \Rb^n; \quad \Qb^{(\kappa)}; \quad \Qb^n_p.
\]
(The group $\Qb_p(\aleph_0)$ is an open subgroup of the homogeneous group $\Qb \otimes \Zb^{\aleph_0}_p$, but the latter is not second-countable.)

There are also well-known structural descriptions of elementary abelian groups, abelian Lie groups, and divisible discrete abelian groups and their duals; these lead to characterizations of the first four families of topologically characteristically simple abelian groups.  The main novelties in the present article are therefore to collect together some known constructions and obstructions of monoliths of the given isomorphism types in compactly generated and/or soluble \lcsc groups (Section~\ref{sec:families}); characterize the remaining family $\Qb_p(\kappa)$ (Proposition~\ref{prop:Qpkappa_characterization}); and observe that this completes the classification of topologically characteristically simple abelian \lcsc groups (Section~\ref{sec:classification}).

Minimal closed normal subgroups, particularly those that are neither compact nor discrete, are of particular interest in the theory of compactly generated locally compact groups.  (Note that if $G$ is a compactly generated locally compact group, there are arbitrarily small compact normal subgroups $N$ such that $G/N$ is second-countable.)  We recall two results from the recent literature.

\begin{thm}[Caprace--Monod, see \cite{CM}]\label{thm:CM}
Let $G$ be a compactly generated locally compact group with no nontrivial compact or discrete normal subgroups.  Then every nontrivial closed normal subgroup of $G$ contains a minimal one.  Moreover, all but finitely many minimal closed normal subgroups of $G$ are abelian.
\end{thm}

Using Theorem~\ref{mainthm}, one can sharpen Theorem~\ref{thm:CM} to the following:

\begin{cor}\label{cor:CM}
Let $G$ be a compactly generated locally compact group with no nontrivial compact or discrete normal subgroups.  Then every nontrivial closed normal subgroup of $G$ contains a minimal one.  Moreover, all but finitely many minimal closed normal subgroups of $G$ are of the form $\Fb_p((t))$, $\Rb^n$ or $\Qb_p(\kappa)$, for some prime $p$, $n \in \Nb$ and $\kappa \in \Nb \cup \{\aleph_0\}$ as applicable.
\end{cor}

Each of the groups $\Fb_p((t))$, $\Rb^n$ and $\Qb_p(\kappa)$ can occur $2^{\aleph_0}$ times as a minimal closed normal subgroup of a compactly generated soluble \lcsc group: see Example~\ref{ex:minimal}.

\begin{thm}[{Reid--Wesolek \cite[Theorem~1.3]{RW}}]
Let $G$ be a compactly generated locally compact group.  Then every finite series of closed normal subgroups of $G$ can be refined to another such series
\[
\{1\} = G_0 < G_1 < \dots < G_n = G,
\]
such that for each $i \in \{1,\dots,n\}$, the factor $G_i/G_{i-1}$ is compact, discrete, or a chief factor of $G$ (that is, $G_i/G_{i-1}$ is a minimal closed normal subgroup of $G/G_{i-1}$).
\end{thm}

Subsequent articles of the author and Wesolek have developed a theory of nonabelian chief factors; the present article is complementary in that it identifies the possible isomorphism types of abelian chief factors.  In particular, one obtains the following structural result for abelian factors of compactly generated locally compact groups.

\begin{cor}
Let $G$ be a compactly generated locally compact group and let $M$ and $N$ be closed normal subgroups such that $M \le N$ and $N/M$ is abelian.  Then there is a finite series
\[
M = G_0 < G_1 < \dots < G_n = N
\]
of closed normal subgroups of $G$, such that for each $0 \le i < n$, the factor $G_{i+1}/G_i$ is compact, discrete, or is a chief factor isomorphic to $\Fb_p((t))$, $\Rb^n$, or $\Qb_p(\kappa)$, for some prime $p$, $n \in \Nb$ and $\kappa \in \Nb \cup \{\aleph_0\}$ as applicable.
\end{cor}

\section{Preliminaries}

\subsection{Notation and terminology}

Throughout, ``\lcsc'' stands for ``locally compact second-countable''.  By convention, we require a locally compact group to be equipped with a \emph{Hausdorff} topology compatible with the group structure.  In particular, it follows from standard results (see \cite[Theorem~8.3]{HR}) that a locally compact group is first-countable if and only if it is metrizable, and second-countable if and only if it is metrizable and $\sigma$-compact.

Automorphisms or isomorphisms of topological groups are required to be homeomorphisms that preserve the group structure.  In particular, a \defbold{characteristic} subgroup is one that is preserved by all automorphisms in the aforementioned sense.

Let $\Nb$ be the set of natural numbers (not including $0$) and let $\Pb$ be the set of prime numbers.

When using a commutative topological ring $R$ as a group, we mean its additive group with the same topology; for the multiplicative group of units of a given ring $R$ we write $R^*$ and equip $R^*$ with the subspace topology, and unless otherwise specified $R^*$ acts on $R$ by multiplication.  Specifically, $\Zb$ and $\Qb$ have the discrete topology while $\Rb$, $\Zb_p$, $\Qb_p$ and $\Fb_p((t))$ have their usual nondiscrete locally compact topologies.

Where convenient we will write abelian groups additively.  Every abelian group $A$ is naturally a $\Zb$-module: given $a \in A$ we define $na$ to be the sum of $n$ copies of $a$ (for $n \ge 0$) or respectively the sum of $-n$ copies of $-a$ (for $n < 0$).

Let $C_{\aleph_0} = \Zb$ and $C_n = \Zb/n\Zb$ for $n \in \Nb$.

The \defbold{(topological) monolith} of a locally compact group $G$ is the intersection $\Mon(G)$ of all nontrivial closed normal subgroups of $G$.  If $\Mon(G)$ is nontrivial, we say that $G$ is \defbold{monolithic}.

Given a sequence $(G_i)_{i \in I}$ of topological groups, the product $\prod_{i \in I}G_i$ is also a group with the product topology.  We note that $\prod_{i \in I}G_i$ is locally compact if and only if all of the groups $G_i$ are locally compact and all but finitely many $G_i$ are compact.
 
Let $(G_i,H_i)_{i \in I}$ be a sequence of pairs of topological groups, where $H_i$ is an open subgroup of $G_i$.  The \defbold{local direct product} $\bigoplus_{i \in I}(G_i,H_i)$ is the set of all tuples $(g_i)_{i \in I}$ such that $g_i \in G_i$ for all $i \in I$ and $g_i \in H_i$ for all but finitely many $i \in I$.  Multiplication is defined pointwise.  The group $\prod_{i \in I}H_i$ is naturally embedded in $\bigoplus_{i \in I}(G_i,H_i)$ as a subgroup; we equip $\bigoplus_{i \in I}(G_i,H_i)$ such that the natural inclusion of $\prod_{i \in I}H_i$ into $\bigoplus_{i \in I}(G_i,H_i)$ is an open embedding.  Note that if $H_i$ is locally compact for all $i \in I$ and compact but for all but finitely many $i \in I$, then $\bigoplus_{i \in I}(G_i,H_i)$ is locally compact.  In the case that $G_i$ is discrete and $H_i$ is trivial for all $i \in I$, then $\bigoplus_{i \in I}(G_i,\{1\})$ simplifies to the direct sum $\bigoplus_{i \in I}G_i$ with the discrete topology.

Given a topological group $G$, write $G^{\kappa}$ for the direct product of $\kappa$ copies of $G$ (or copies indexed by $\kappa$), with the product topology; when $G$ is discrete, write $G^{(\kappa)}$ for the direct sum of $\kappa$ copies of $G$ (or copies indexed by $\kappa$) with the discrete topology.  Note that given $n \in \Nb$ and cardinals $\kappa,\kappa'$ we have
\[
G^{\kappa+n} \times G^{(\kappa')} \cong G^{\kappa} \times G^{(\kappa'+n)};
\]
we will use this isomorphism without further comment when indexing local direct powers of discrete groups.

Given an abelian topological group $U$, a \defbold{(topological) direct summand} is a subgroup $V$ of $U$ such that $U$ admits a direct sum decomposition $U = V \oplus W$ as a topological group, meaning that both $V$ and $W$ are closed subgroups and $U$ carries the product topology.

\subsection{Compact elements of abelian groups}

\begin{thm}[{\cite[Theorem~25]{Morris}}]\label{thm:open_subgroup}
Every locally compact abelian group has an open subgroup of the form $\Rb^a \oplus K$, where $a \in \Nb \cup \{0\}$ and $K$ is compact.
\end{thm}

Given a locally compact abelian group $G$, we say $g \in G$ is \defbold{compact} if $\cgrp{g}$ is compact.  Write $P(G)$ for the set of compact elements.  For a prime $p$, write $P_p(G)$ for the set of $g \in G$ such that $g^{p^n} \rightarrow 0$ as $n \rightarrow \infty$.  Note that $P_p(G) \le P(G)$ for all $p \in \Pb$. 

\begin{lem}\label{lem:PG_closed}
Let $G$ be a locally compact abelian group and let $p \in \Pb$.  Then $P(G)$ is a closed subgroup of $G$.  If $G$ is connected, then $P(G)$ is compact.  If $G$ is totally disconnected, then $P_p(G)$ is closed and $qP_p(G) = P_p(G)$ for all $q \in \Pb \setminus \{p\}$.
 \end{lem}
 
\begin{proof}
To show that $P(G)$ is closed, it is enough to show that $P(G) \cap U$ is closed for some open subgroup $U$ of $G$.  By Theorem~\ref{thm:open_subgroup} we can take $U = \Rb^a \oplus K$ for $K$ compact.  We then see that $P(G) \cap U = K$, so $P(G)$ is closed; if $G$ is connected then $G = U$, so $P(G) = K$, and hence $P(G)$ is compact.

In a finite abelian group $F$, we see that $P_p(F)$ is the $p$-Sylow subgroup of $F$, and $qP_p(F) = P_p(F)$ for all $q \in \Pb \setminus \{p\}$.  The same is then true when $F$ is a profinite abelian group, by an inverse limit argument.  In particular, $P_p(F)$ is closed.
 
If $G$ is totally disconnected, then $G$ has a profinite open subgroup $U$ and we have $P_p(G) \cap U = P_p(U)$, so $P_p(G)$ is closed.  Given $q \in \Pb \setminus \{p\}$, for all profinite open subgroups $K$ of $P_p(G)$ we have $qK = K$; since $P_p(G)$ is the union of such subgroups, we have $qP_p(G) = P_p(G)$.
\end{proof}

If $A$ is a locally compact abelian group with $P_p(A) = A$, then we can naturally regard $A$ as a $\Zb_p$-module: specifically, given $\lambda_n \in \Zb$ converging to some $\lambda \in \Zb_p$, we define $\lambda a = \lim_{n \rightarrow \infty}\lambda_n a$.  The limit exists exactly because $a \in P_p(A)$; what is more, one sees that the multiplication map $\Zb_p \times A \rightarrow A$ is continuous.  In particular, for all $a \in A$ the closed subgroup $\cgrp{a}$ is the same as the cyclic $\Zb_p$-submodule $\Zb_p a$.

\subsection{Divisibility}

We say $a \in A$ is \defbold{divisible} if for all $n \in \Nb$ there exists an $n$-th root of $a$, that is, $b \in A$ such that $nb=a$, and \defbold{$p$-divisible} for $p \in \Pb$ if for all $n \ge 0$ there exists $b \in A$ such that $p^nb = a$.  A \defbold{divisible group} (respectively, \defbold{$p$-divisible group}) is one in which every element is divisible (respectively, $p$-divisible).

A subgroup $B$ of an abelian group $A$ is \defbold{essential} if $B$ intersects nontrivially every nontrivial subgroup of $A$.  Note that in the case that $A$ is torsion-free, $B$ is essential if and only if $A/B$ is a torsion group.

(See \cite[P.31]{Armacost}.)  Every locally compact abelian group $A$ has a largest divisible subgroup, $\divi(A)$, and a minimal divisible extension, $A^{\divi}$, such that $A^{\divi}$ is again a locally compact group into which $A$ embeds as an open subgroup.  The minimal divisible extension is characterized up to isomorphism by the properties that $A^{\divi}$ is divisible and $A$ is an open essential subgroup of $A^{\divi}$.  In particular, if $U$ is an open essential subgroup of $A$, then $A^{\divi} = U^{\divi}$.

If $A$ is torsion-free, then roots are unique (if they exist), so we can write $a/n$ for the $n$-th root of $a \in A$.  There are more direct descriptions of $\divi(A)$ and $A^{\divi}$ in this case.

\begin{lem}\label{lem:tf_div}
Let $A$ be a torsion-free abelian group.
\begin{enumerate}[(i)]
\item $\divi(A)$ is the set of all divisible elements of $A$.
\item $A^{\divi}$ is the tensor product $\Qb \otimes A$ of $\Zb$-modules, where $A$ is embedded via $a \mapsto (1,a)$ and $\Qb \otimes A$ is equipped with the topology extending that of $A$.
\item Given a continuous homomorphism $\theta: A \rightarrow B$ between torsion-free abelian groups, then $\theta$ has a unique extension to a continuous homomorphism $\theta^{\divi}: A^{\divi} \rightarrow B^{\divi}$.
\end{enumerate}
\end{lem}

\begin{proof}
(i) The set $D$ of divisible elements is clearly a subgroup; it is enough to show that $D$ is also closed under taking roots.  Given $a \in D$ and $m,n \in \Nb$, then by uniqueness of roots, $a/(mn)$ is the $n$-th root of $a/m$; by letting $n$ vary we conclude that $a/m \in D$.

(ii) The tensor product $\Qb \otimes A$ consists of pure tensors: given $q_1,q_2,r_1,r_2 \in \Zb \setminus \{0\}$ and $a_1,a_2 \in A$, we have
\[
\left(\frac{q_1}{r_1},a_1 \right) + \left(\frac{q_2}{r_2},a_2 \right) = \left( \frac{1}{r_1r_2}, q_1r_2a_1 + q_2r_1a_2 \right). 
\]
From here it is easy to see that $\Qb \otimes A$ is divisible and torsion-free, that $(\Qb \otimes A)/A$ is torsion and that the map $a \mapsto (1,a)$ is an injective homomorphism.  Thus $A$ is essential in $\Qb \otimes A$.  We can then take the topology of $\Qb \otimes A$ such that $a \mapsto (1,a)$ is an open embedding, with the result that $\Qb \otimes A$ is the minimal divisible extension of $A$ in the topological sense.

(iii) We use tensor products as in (ii).  It is clear that the only possible extension of $\theta$ is to define $\theta^{\divi}(1/n,a) = (1/n,\theta(a))$ for all $n \in \Nb$ and $a \in A$. The uniqueness of roots ensures that this is a well-defined homomorphism from $\Qb \otimes A$ to $\Qb \otimes B$.  The continuity of $\theta^{\divi}$ follows from the fact that it is continuous in a neighbourhood of the identity.
\end{proof}

The discrete divisible abelian groups have been classified.

\begin{lem}[{\cite[Theorem~A.14]{HR}}]\label{lem:discrete_divisible}
Let $A$ be a discrete abelian group.  Then the following are equivalent:
\begin{enumerate}[(i)]
\item $A$ is divisible;
\item $A \cong \Qb^{(\kappa)} \oplus \bigoplus_{p \in \Pb}(\Zb[\frac{1}{p}]/\Zb)^{(\kappa_p)}$ for some cardinals $\kappa$ and $\kappa_p$.
\end{enumerate}
\end{lem}

The next two lemmas give some useful conditions under which $A$ is \defbold{densely divisible}, meaning that $\divi(A)$ is a dense subgroup of $A$.

\begin{lem}[{See \cite[P.28(b)]{Armacost}}]\label{lem:compact_divisible}
Let $A$ be a compact abelian group.  Then the following are equivalent:
\begin{enumerate}[(i)]
\item $A$ is divisible;
\item $A$ is densely divisible;
\item $A$ is connected.
\end{enumerate}
\end{lem}
 
 \begin{lem}\label{lem:pdense_divisible}
 Let $A$ be a torsion-free totally disconnected locally compact abelian group such that $P_p(A) = A$.  Then $\divi(A)$ is dense in $A$ if and only if $pA$ is dense in $A$.
 \end{lem}
 
 \begin{proof}
 If $\divi(A)$ is dense in $A$ then certainly $pA$ is dense in $A$, since $pA \ge p\divi(A) = \divi(A)$.  So we may suppose that $pA$ is dense in $A$.
 
Let $U$ be a compact open subgroup of $A$; the quotient $A/U$ is then a discrete $p$-group.  Since $pA$ is dense in $A$, we see that $p(A/U) = A/U$, that is, $A/U$ is $p$-divisible.  Since $A/U$ is a $p$-group, we conclude that ${A/U \cong (\Zb[\frac{1}{p}]/\Zb)^{(\kappa)}}$.  In particular, for all $a \in A$ there is a homomorphism ${f_a: \Zb[\frac{1}{p}] \rightarrow A/U}$ such that $a+U$ is contained in the image of $f_a$.  By \cite[Theorem~1.34]{HHR}, $f_a$ lifts to a homomorphism ${\tilde{f}_a: \Zb[\frac{1}{p}] \rightarrow A}$, such that there exists $\tilde{a} \in a+U$ in the image of $\tilde{f}_a$.  In particular, we see that the element $\tilde{a}$ is $p$-divisible in $A$.  Since $P_p(A) = A$, for all $n \ge 0$ the $p^n$-th root $\tilde{a}/p^n$ of $\tilde{a}$ belongs to a pro-$p$ subgroup of $A$; in particular, $\tilde{a}/p^n$ has an $m$-th root for all $m$ coprime to $p$.  Hence in fact $\tilde{a}$ is divisible in $A$, that is, $\tilde{a} \in \divi(A)$.  Thus $A = \divi(A)+U$.  Since $U$ can be made arbitrarily small, we conclude that $\divi(A)$ is dense in $A$.
\end{proof}
 
\subsection{Pontryagin duality}

Let $\Tb$ denote the multiplicative group of complex numbers of absolute value $1$.  Given a locally compact abelian group $A$, a \defbold{character} of $A$ is a continuous homomorphism from $A$ to $\Tb$.  The set $\widehat{A}$ of characters, equipped with pointwise multiplication and the compact-open topology, is called the \defbold{character group} or \defbold{dual} of $A$; it is again a locally compact abelian group.  We recall some standard facts about the relationship between $A$ and $\widehat{A}$; see for example \cite[Chapter 6]{HR} for details.

\begin{thm}[{Pontryagin--van Kampen; see \cite[Theorem 24.8]{HR}}]
There is a natural isomorphism of topological groups between a locally compact abelian group $A$ and the double dual of $A$, namely $a \in A$ corresponds to the character $\chi_a$ of $\widehat{A}$, where $\chi_a(f) = f(a)$ for all $f \in \widehat{A}$.
\end{thm}

As such we will simply identify $A$ with its double dual.

\begin{lem}\label{lem:dual}Let $A$ be a locally compact abelian group.
\begin{enumerate}[(i)]
\item\emph{({\cite[Theorem~29]{Morris}})} $A$ is metrizable if and only if $\widehat{A}$ is $\sigma$-compact.  In particular, the class of abelian \lcsc groups is closed under taking duals.
\item\emph{({\cite[Theorem~23.17]{HR}})} $A$ is compact if and only if $\widehat{A}$ is discrete.
\item\emph{({Robertson \cite[Theorem~5.2]{Robertson}})} $A$ is torsion-free if and only if $\widehat{A}$ is densely divisible.
\end{enumerate}
\end{lem}

(See \cite[\S24]{HR}.) For every continuous homomorphism $\theta: A \rightarrow B$ of locally compact abelian groups there is a continuous homomorphism $\widehat{\theta}: \widehat{B} \rightarrow \widehat{A}$, called the \defbold{adjoint} or \defbold{dual} of $\theta$, given by $\widehat{\theta}(\chi) = \chi \circ \theta$.  The double adjoint recovers the original homomorphism, and the adjoint of an isomorphism is an isomorphism.  In particular, there is a natural isomorphism between $\Aut(A)$ and $\Aut(\widehat{A})$ given by $\theta \mapsto \widehat{\theta}^{-1}$.

\begin{lem}[{\cite[\S23.33]{HR}}]\label{lem:localdirect_dual}
Let $G = \bigoplus_{i \in I}(G_i,H_i)$ be a local direct product of locally compact abelian groups, where $H_i$ is a compact open subgroup of $G_i$ for all $i \in I$.  Then
\[
\widehat{G} = \bigoplus_{i \in I}(\widehat{G_i},H^{\perp}_i),
\]
where $H^{\perp}_i :=  \{ \chi \in \widehat{G_i} \mid \chi(H_i) = \{1\}\}$.
\end{lem}

\subsection{Torsion-free compact abelian groups}

Torsion-free compact abelian groups have been classified: they are exactly the duals of divisible discrete abelian groups.

\begin{lem}[{\cite[Theorem 25.8]{HR}}]\label{lem:tf_compact}
The compact torsion-free abelian groups are exactly the topological groups of the form 
\[
\widehat{\Qb}^{\kappa} \oplus \prod_{p \in \Pb}\Zb^{\kappa_p}_p
\]
for some cardinals $\kappa$ and $\kappa_p$.
\end{lem}

Given a torsion-free abelian pro-$p$ group $U$ and a subset $X = \{e_i \mid i \in I\}$ of $U$, we say $X$ is a \defbold{basis} for $U$ if $U$ is freely generated by $X$ as an abelian pro-$p$ group; in other words, $X$ converges to $0$ in $U$ (that is, every neighbourhood of $0$ in $U$ contains all but finitely many elements of $X$), and every element of $U$ is expressible as a sum $\sum_{i \in I}\lambda_i e_i$ with unique coefficients $\lambda_i \in \Zb_p$ (here the infinite sum makes sense because $X$ converges to $0$).  We can then write $U = \prod_{i \in I}\cgrp{e_i}$.  Then \defbold{rank} of $U$ (as a pro-$p$ group) is the cardinality of some basis for $U$.  A subgroup $H$ of $U$ is \defbold{pure} if for all $n \in \Nb$ we have $H \cap nU = nH$; in other words, every element of $H$ with an $n$-th root in $U$ has an $n$-th root in $H$.
 
 \begin{lem}\label{lem:Zp_direct}
 Let $U$ be a torsion-free abelian pro-$p$ group.
 \begin{enumerate}[(i)]
 \item $U$ is free in the category of abelian pro-$p$ groups, in other words, $U$ has a basis.  In particular, $U$ is second-countable if and only if $U \cong \Zb^{\kappa}_p$ for some $\kappa \le \aleph_0$.
 \item Every closed pure subgroup of $U$ is a direct summand.
 \item Any two bases for $U$ are conjugate under the action of $\Aut(U)$; in particular, the rank is well-defined.
 \item Every $\Zb_p$-submodule of $U$ of finite rank is closed.
 \end{enumerate}
 \end{lem}
 
 \begin{proof}
(i) We have $U \cong \Zb^\kappa_p$ for some cardinal $\kappa$ by Lemma~\ref{lem:tf_compact}.  We see that $\Zb^\kappa_p$ is second-countable if and only if $\kappa \le \aleph_0$.

(ii) Given a closed pure subgroup $V$, then $U/V$ is a free abelian pro-$p$ group by (i), and hence projective in the category of abelian pro-$p$ groups (see \cite[Proposition~5.2.2]{Wilson}).  Thus the quotient map $U \rightarrow U/V$ splits, in other words, $V$ is a direct summand.

(iii) If $U \cong \Zb^{\kappa}_p$ then $\widehat{U} \cong (\Zb[\frac{1}{p}]/\Zb)^{(\kappa)}$, and we can recover $\kappa$ as the largest size of an independent set in $(\Zb[\frac{1}{p}]/\Zb)^{(\kappa)}$; thus the rank is well-defined.  The fact that bases are conjugate then follows by the standard universal property argument for free objects; see for instance \cite[Proposition~5.1.2]{Wilson}.
 
(iv) A $\Zb_p$-submodule $V$ of $U$ of finite rank is one of the form $\{\sum^n_{i=1}\lambda_iv_i \mid \lambda_i \in \Zb_p\}$ for some $v_1,\dots,v_n \in V$.  In other words, $V$ is the sum in $U$ of the compact subgroups $\Zb_p v_i = \cgrp{v_i}$ for $1 \le i \le n$.  Since a sum of finitely many compact subgroups is compact, it follows that $V$ is compact and hence closed.
 \end{proof}

\section{Five families of abelian groups}\label{sec:families}

We consider five families of abelian \lcsc groups.  For each family, we give a characterization in terms of standard properties of locally compact groups; in no case do we need to include ``topologically characteristically simple'' as one of the properties.  However, all the families indeed consist of topologically characteristically simple groups, which we demonstrate by giving examples of \lcsc groups $G$ whose monolith is the abelian group of interest.  In most cases $G$ will be compactly generated; for the exceptions we explain why no such $G$ exists.

\subsection{Elementary abelian groups}

A group is \defbold{elementary abelian} if it is abelian and has prime exponent.  The elementary abelian \lcsc groups have a well-known structure.

\begin{prop}[{see e.g. \cite[Corollary~3.26]{HHR}}]\label{prop:elab_characterization}
Let $A$ be a nontrivial abelian \lcsc group and let $p \in \Pb$.  Then the following are equivalent:
\begin{enumerate}[(i)]
\item $A \cong C^{\kappa}_p \oplus C^{(\kappa')}_p$ for some $\kappa \in \Nb \cup \{\aleph_0\}$ and $\kappa' \in \{0,\aleph_0\}$;
\item $A$ is isomorphic to one of: $C^{\aleph_0}_p$, $C^{(\aleph_0)}_p$, $C^{\aleph_0}_p \oplus C^{(\aleph_0)}_p$, or $C^n_p$ for some $n \in \Nb$;
\item $A$ has exponent $p$.
\end{enumerate}
\end{prop}

All elementary abelian \lcsc $p$-groups occur as minimal closed normal subgroups of soluble compactly generated \lcsc groups; to show this, we first prove a lemma about wreath products.

Let $F$ be a finite group.  The restricted wreath product $F \wr \Zb$ is a finitely generated discrete group of the form $F^{(\Zb)} \rtimes \Zb$.  The unrestricted wreath product $F \Wr \Zb$ instead has base $F^{\Zb}$ with the product topology.  In each case we think of elements of the base as functions from $\Zb$ to $F$ in the obvious manner.

\begin{lem}\label{lem:shift_monolith}Let $F$ be a finite group with a noncentral monolith $M$.   Then
\[
\Mon(F \Wr \Zb) = M^{\Zb}; \quad \Mon(F \wr \Zb) = M^{(\Zb)}.
\]
\end{lem}

\begin{proof}
Let $G$ be $F \Wr \Zb$ or $F \wr \Zb$.  Given $a \in F$, let $\delta_{i,a}$ be the function from $\Zb$ to $F$ such that $\delta_{i,a}(i) = a$ and $\delta_{i,a}(j) = 0$ for $j \neq i$.  Note that the hypotheses on $F$ ensure that $F$ has trivial centre.

Let $N$ be a nontrivial closed normal subgroup of $G$.  Then $N$ contains some nontrivial element $(f,i)$ where $f$ is in the base and $i \in \Zb$.  If $i \neq 0$ we see that $N$ contains all elements of $F^{(\Zb)}$ of the form $\delta_{j,a} - \delta_{j-i,a}$, so we may assume $f$ is nontrivial.  Then $f(j) \neq 1$ for some $j \in \Zb$, and hence $N$ contains the nontrivial element $\delta_{j,[a,f(j)]}$ of $F^{(\Zb)}$, where $a$ is any element of $F \setminus \CC_F(f(j))$.  Thus $N$ contains an element of the form $\delta_{j,b}$ where $b \in F \setminus \{1\}$.  By conjugating in the $j$-th copy of $F$, we deduce that $N$ contains $\delta_{j,c}$ where $c$ ranges over a nontrivial normal subgroup of $F$, and in particular for all $c \in M$.  By applying the shift map, in fact $N$ contains $\delta_{k,c}$ for all $k \in \Zb$ and $c \in M$, that is, $M^{(\Zb)} \le N$.  Since $N$ is closed we have $\ol{M^{(\Zb)}} \le N$.  Conversely, it is easy to see that $\ol{M^{(\Zb)}}$ is itself a nontrivial closed normal subgroup of $G$.  Thus $\Mon(G) = \ol{M^{(\Zb)}}$.\end{proof}

\begin{prop}\label{prop:elab_minimal}
Let $A$ be an elementary abelian \lcsc group.  Then $A$ is topologically characteristically simple.  Moreover, there is a compactly generated soluble \lcsc group $G$ of derived length at most $3$ such that $\Mon(G) \cong A$.
\end{prop}

\begin{proof}
We go through the possibilities in Proposition~\ref{prop:elab_characterization}(ii).

\emph{Case 1: $A$ is finite.}

In this case $A \cong C^n_p$ for some $n \in \Nb$, and we can regard $A$ as the additive group of the field $\Fb_{p^n}$ with $p^n$ elements.  The multiplicative group $\Fb^*_{p^n}$ is then an abelian group that acts transitively on the nonzero elements of $\Fb_{p^n}$, and so $A$ is isomorphic to the monolith of the finite metabelian group $\Fb_{p^n} \rtimes \Fb^*_{p^n}$ (or just $\Fb_2$ in the case $(n,p) = (1,2)$).

\emph{Case 2: $A$ is isomorphic to $C^{\aleph_0}_p$ or $C^{(\aleph_0)}_p$.}

We obtain $A$ as the monolith of a group of the form $F \Wr \Zb$ or $F \wr \Zb$, respectively, for some finite group $F$; by Lemma~\ref{lem:shift_monolith}, it is enough to take a finite metabelian group $F$ with noncentral monolith isomorphic to $C^n_p$ for some $n \ge 1$.  For example, we can take $F = \Fb_{p^2} \rtimes \Fb^*_{p^2}$.

\emph{Case 3: $A \cong C^{\aleph_0}_p \oplus C^{(\aleph_0)}_p$.}

In this case $A$ is isomorphic to the additive group of $\Fb_p((t))$; write $\cdot$ for multiplication in the field $\Fb_p((t))$.  We consider the compactly generated metabelian \lcsc group $G = \Fb_p((t)) \rtimes \grp{s}$, where $sgs^{-1} = t\cdot g$ for all $g \in \Fb_p((t))$.  Let $N$ be a nontrivial closed normal subgroup of $G$; it is easy to see that $N$ intersects $\Fb_p((t))$ nontrivially, so let us assume $N \le \Fb_p((t))$.  Given a nontrivial element $g$ of $N$, we see that for all $f \in \Fb_p[t,t^{-1}]$, the product $f \cdot g$ is a sum of conjugates of $g$ in $G$, and hence an element of $N$.  Since $\Fb_p[t,t^{-1}]$ is dense in $\Fb_p((t))$ and $N$ is closed, in fact $f \cdot g \in N$ for all $f \in \Fb_p((t))$, that is, $N$ is a nonzero ideal of $\Fb_p((t))$.  Since $\Fb_p((t))$ is a field we conclude that $N = \Fb_p((t))$.  Thus $A$ is the monolith of $G$.
\end{proof}

\subsection{The groups $\Rb^n$}

The characterization of the groups $\Rb^n$ is well-known.

\begin{prop}\label{prop:R_characterization}Let $A$ be a nontrivial abelian locally compact group.  Then the following are equivalent:
\begin{enumerate}[(i)]
\item $A \cong \Rb^n$ for some $n \in \Nb$ (in particular, $A$ is second-countable);
\item $A$ is connected and $P(A) = \{0\}$.
\end{enumerate}
\end{prop}

\begin{proof}
Clearly $\Rb^n$ is connected, abelian and $P(\Rb^n) = \{0\}$.  Conversely, suppose $A$ is a connected abelian \lcsc group such that $P(A) = \{0\}$.  Then $A$ has no proper open subgroup and no nontrivial compact subgroup, so $A \cong \Rb^n$ for some $n \in \Nb$ by Theorem~\ref{thm:open_subgroup}.
\end{proof}

We note also that $\Rb^m \cong \Rb^n$ if and only if $m=n$ (see \cite[Theorem~9.12]{HR}).

\begin{example}\label{ex:Rn}Fix $n \in \Nb$.  We can realize $\Rb^n$ as the monolith of the virtually connected soluble Lie group $(\Rb \rtimes \Rb^*_{>0}) \wr C_n$.\end{example}

\subsection{The groups $\Qb^{(\kappa)}$}

The structure of divisible countable torsion-free abelian groups is well-known, but the circumstances in which such groups can occur as minimal closed normal subgroups are more interesting.

\begin{prop}\label{prop:Q_characterization}Let $A$ be a nontrivial abelian \lcsc group.  Then the following are equivalent:
\begin{enumerate}[(i)]
\item $A \cong \Qb^{(\kappa)}$ for some $\kappa \in \Nb \cup \{\aleph_0\}$;
\item $A$ is discrete, torsion-free and divisible;
\item $A$ is totally disconnected, $pA$ is dense in $A$ for all $p \in \Pb$, and $P(A) = \{0\}$.
\end{enumerate}
\end{prop}

\begin{proof}
(i) and (ii) are equivalent by Lemma~\ref{lem:discrete_divisible}; note that we can have at most countably many copies of $\Qb$, since $A$ is second-countable.  Clearly (ii) implies (iii).

Suppose that $A$ is totally disconnected, $pA$ is dense in $A$ for all $p \in \Pb$, and $P(A) = \{0\}$.  We see that $A$ has no nontrivial compact subgroups; by Theorem~\ref{thm:open_subgroup} it follows that $A$ is discrete, and clearly also $A$ is torsion-free.  We also have $pA = A$ for all $p \in \Pb$.  It is then clear (by induction on the prime factorization of $n$) that $nA = A$ for all $n \in \Nb$; hence $A$ is divisible.  Thus (iii) implies (ii), completing the proof that (i)--(iii) are equivalent.
\end{proof}

If $A \cong \Qb^{(\kappa)}$ for $\kappa \in \Nb \cup \{\aleph_0\}$, then $\kappa$ is the supremum of the number of generators of any finitely generated subgroup of $A$.  In particular, $\kappa$ is uniquely determined.

For whether or not $\Qb^{(\kappa)}$ appears as a minimal normal subgroup of a compactly generated (soluble) group, we appeal to an article \cite{Hall} by P. Hall.

\begin{thm}[{Hall, \cite[Theorem~2]{Hall}}]\label{thm:Hall}There is a finitely generated soluble group $G$ of derived length 3 such that $\Mon(G) \cong \Qb^{(\aleph_0)}$.\end{thm}

More specifically, writing $\Qb^{(\aleph_0)}$ as a $\Qb$-vector space $\Qb^{(\Zb)}$ with basis $\{v_m \mid m \in \Zb\}$, Hall's construction is
\[
G = \Qb^{(\Zb)} \rtimes \grp{\eta,\xi},
\]
where $\xi$ and $\eta$ act on $\Qb^{(\Zb)}$ via $\xi.v_m = v_{m+1}$ and $\eta.v_m = \rho(m)v_m$, where $\rho$ is some bijection from $\Zb$ to $\Pb$.  One sees that the conjugates of $\eta$ in $\grp{\eta,\xi}$ commute with one another, so $\grp{\eta,\xi} \cong \Zb \wr \Zb$.

By contrast, Hall shows that for $n \in \Nb$, the group $\Qb^n$ cannot occur as a minimal normal subgroup of a finitely generated group.  We adapt the argument to show that $\Qb^n$ cannot occur as a minimal closed normal subgroup of a compactly generated group, which leads to another characterization of the group $\Qb^{(\aleph_0)}$.

\begin{prop}\label{prop:Q_characterization:minimal}
The group $\Qb^{(\aleph_0)}$ is the only discrete, abelian and torsion-free group that can occur as a minimal closed normal subgroup of a compactly generated \lcsc group.
\end{prop}

\begin{proof}
The group is $\Qb^{(\aleph_0)}$ is discrete, abelian and torsion-free, and we have seen that it can occur as a minimal closed normal subgroup of a compactly generated \lcsc group.

Now suppose that $G$ is a compactly generated \lcsc group and that $M$ is a minimal closed normal subgroup of $G$ that is discrete, abelian and torsion-free.  In particular, $M$ is characteristically simple.  Since $M$ is torsion-free, we must have $nM = M$ for all $n \in \Nb$, that is, $M$ is divisible.

We may assume for a contradiction that $M$ is not isomorphic to $\Qb^{(\aleph_0)}$, so by Proposition~\ref{prop:Q_characterization} we have $M \cong \Qb^n$ for some $n \in \Nb$.  Let $\{v_1,\dots,v_n\}$ be a basis for $M$ as a $\Qb$-vector space.  Since $M$ is countable, the group $\CC_G(v_i)$ is a closed subgroup of $G$ of countable index, and hence $\CC_G(v_i)$ is open for $1 \le i \le n$.  Moreover, we observe that since elements of $M$ are uniquely divisible, $\CC_G(v_i)$ centralizes all rational multiples of $v_i$.  Consequently, $\CC_G(M)$ is the following open subgroup of $G$:
\[
\CC_G(M) = \bigcap^n_{i=1}\CC_G(v_i).
\]
Clearly also $\CC_G(M)$ is normal; the quotient $G/\CC_G(M)$ is then a finitely generated group, so we can write $G = \grp{x_1,\dots,x_m}\CC_G(M)$.

The rest of the argument is as in \cite{Hall}.  Writing $M$ additively for the moment, and letting $x_1,\dots,x_m$ act on $M$ by left conjugation, we have equations of the form
\[
x_{\alpha}.v_i = \sum^n_{j=1} \xi_{ij\alpha}v_j; \; x^{-1}_{\alpha}.v_i = \sum^n_{j=1} \eta_{ij\alpha}v_j,
\]
where $i$ runs from $1$ to $n$ and $\alpha$ from $1$ to $m$.  The coefficients $\xi_{ij\alpha}$ and $\eta_{ij\alpha}$ are rational numbers, which between them involve only a finite set of primes $\pi$ in their denominators.  Let $M_1$ be the group generated by all $G$-conjugates of $v_1,\dots,v_n$.  Then every element of $M_1$ is of the form $\sum^n_{j=1}\lambda_jv_j$, where each of the coefficients $\lambda_j$ belongs to the proper subring $\Zb[\frac{1}{p}: p \in \pi]$ of $\Qb$.  In particular, $M_1$ is a proper nontrivial subgroup of $M$ that is normal in $G$, contradicting the minimality of $M$.
\end{proof}

On the other hand, $\Qb^n$ does occur as the monolith of a countable soluble group $G$.  For example, similar to Example~\ref{ex:Rn}, we can take $G = (\Qb \rtimes \Qb^*_{>0}) \wr C_n$.  The difference is that $\Rb^*_{>0}$ is compactly generated whereas $\Qb^*_{>0}$ is not.

\subsection{The groups $\widehat{\Qb}^\kappa$}

The character group $\widehat{\Qb}$ of $\Qb$ is perhaps less widely-known than $\Qb$ itself, but still a basic object in the theory of locally compact abelian groups.  It can be constructed as follows (see for instance \cite{Conrad} for details):

Given a $p$-adic number $a$, write $\{a\}_p$ for the $p$-adic fractional part of $a$, in other words $\{a\}_p$ is the unique element $r \in \Zb[\frac{1}{p}]$ such that $0 \le r < 1$ and $a \in r + \Zb_p$.

Let $\Qb_{\infty} = \Zb_{\infty} = \Rb$.  Let $\mathbb{A}$ be the ring of adeles over $\Qb$: specifically, 
\[
\mathbb{A} = \bigoplus_{p \in \Pb \cup \{\infty\}} (\Qb_p,\Zb_p).
\]
We identify $\mathbb{A}$ with its additive group, which is a locally compact abelian group.  For each $a \in \mathbb{A}$ there is a character $\xi_a$ of $\Qb$ defined by
\[
\xi_a(r) = e^{-2\pi i ra_{\infty}} \prod_{p \in \Pb}e^{2\pi i \{ra_p\}_p}.
\]
The map $a \mapsto \xi_a$ defines a quotient map $\theta: \mathbb{A} \rightarrow \widehat{\Qb}$.  The kernel consists of those $a \in \mathbb{A}$ such that $a_p$ is a constant rational number over all $p \in \Pb \cup \{\infty\}$; in particular, $\ker\theta \cong \Qb$ as a discrete subgroup of $\mathbb{A}$.  Thus $\widehat{\Qb}$ can be identified with the quotient $\mathbb{A}/\Qb$.

We can characterize the groups $\widehat{\Qb}^\kappa$ in a similar way to their dual groups $\Qb^{(\kappa)}$.

\begin{prop}\label{prop:Qhat_characterization}Let $A$ be a nontrivial abelian \lcsc group.  Then the following are equivalent:
\begin{enumerate}[(i)]
\item $A \cong \widehat{\Qb}^{\kappa}$ for some $\kappa \in \Nb \cup \{\aleph_0\}$;
\item $A$ is compact, torsion-free and densely divisible;
\item $A$ is torsion-free and connected, and $P(A) = A$.
\end{enumerate}
\end{prop}

\begin{proof}
Suppose $A \cong \widehat{\Qb}^{\kappa}$ for some $\kappa \in \Nb \cup \{\aleph_0\}$.  Then by Lemma~\ref{lem:localdirect_dual}, $A$ is the character group of $\Qb^{(\kappa)}$.  Since $\Qb^{(\kappa)}$ is discrete and torsion-free, it follows by Lemma~\ref{lem:dual} that $A$ is compact and densely divisible.  Thus (i) implies (ii).

Suppose (ii) holds.  Clearly $P(A) = A$, and by Lemma~\ref{lem:compact_divisible}, $A$ is connected.  Thus (ii) implies (iii).

Suppose (iii) holds.  By Lemma~\ref{lem:PG_closed}, $A$ is compact.  We deduce that (i) holds from Lemma~\ref{lem:tf_compact} and the fact that $A$ is second-countable.  This completes the proof that (i)--(iii) are equivalent.
\end{proof}

If $A \cong \widehat{\Qb}^{\kappa}$, then the cardinal $\kappa$ is uniquely determined by $\widehat{A}$, and hence by $A$.

Using the natural isomorphism between $\Aut(\Qb^{(\kappa)})$ and $\Aut(\widehat{\Qb}^{\kappa})$, Hall's construction from the last subsection immediately gives rise to a compactly generated soluble \lcsc group $G$ of the form
\[
G = \widehat{\Qb}^{\aleph_0} \rtimes (\Zb \wr \Zb)
\]
with $\Mon(G) \cong \widehat{\Qb}^{\aleph_0}$; similarly for $n \in \Nb$, we obtain $\widehat{\Qb}^n$ as the monolith of a soluble \lcsc group of the form
\[
(\widehat{\Qb} \rtimes \Qb^*_{>0}) \wr C_n.
\]

On the other hand, the same obstacle as for $\Qb^n$ prevents $\widehat{\Qb}^n$ from occurring as a minimal closed normal subgroup of a compactly generated \lcsc group and gives a characterization of $\widehat{\Qb}^{\aleph_0}$.

\begin{prop}\label{prop:Qhat_characterization:minimal}
The group $\widehat{\Qb}^{\aleph_0}$ is the only compact, abelian and torsion-free group that can occur as a minimal closed normal subgroup of a compactly generated \lcsc group.
\end{prop}

\begin{proof}
The group is $\widehat{\Qb}^{\aleph_0}$ is compact, abelian and torsion-free, and we have seen that it can occur as a minimal closed normal subgroup of a compactly generated \lcsc group.

Now suppose that $G$ is a compactly generated \lcsc group and that $M$ is a minimal closed normal subgroup of $G$ that is compact, abelian and torsion-free.  In particular, $M$ is characteristically simple, so as a compact group, $M$ must either be connected or pro-$p$ for some $p$.  The latter case can be ruled out: if $M$ were to be a torsion-free pro-$p$ group, then $pM$ would be a proper nontrivial closed characteristic subgroup of $M$, which would contradict the minimality of $M$.  So $M$ is connected.

We may assume for a contradiction that $M$ is not isomorphic to $\widehat{\Qb}^{\aleph_0}$, so by Proposition~\ref{prop:Qhat_characterization} we have $M \cong \widehat{\Qb}^n$ for some $n \in \Nb$, and hence $\widehat{M} \cong \Qb^n$.  Let $\{\chi_1,\dots,\chi_n\}$ be a basis for $\widehat{M}$ as a $\Qb$-vector space.  We have a natural action of $G$ on $\widehat{M}$ given by
\[
(g.\chi)(m) = \chi(g\inv mg) \quad (g \in G, \chi \in \widehat{M}, m \in M).
\]
Since $\widehat{M}$ is countable, the stabilizer of $\xi_i$,
\[
G_i := \{g \in G \mid \forall m \in M: \chi_i(g\inv mg) = \chi_i(m)\}
\]
has countable index in $G$.  Moreover, $G_i$ is closed: specifically, we have
\[
G_i/\ker\chi_i = \CC_H(M/\ker\chi_i)
\]
where $H = \N_G(\ker\chi_i)/\ker\chi_i$, and we see that $H$ is a locally compact group since both $\ker\chi_i$ and $\N_G(\ker\chi_i)$ are closed.  Thus $G_i$ is open.  As in the proof of Proposition~\ref{prop:Q_characterization:minimal} we see that $\bigcap^n_{i=1}G_i$ acts trivially on $\widehat{M}$.  Since the characters of $M$ distinguish points, we have
\[
\CC_G(M) = \bigcap^n_{i=1}G_i,
\]
so $\CC_G(M)$ is open and $G = \grp{x_1,\dots,x_m}\CC_G(M)$.

As in the proof of Proposition~\ref{prop:Q_characterization:minimal}, we obtain a proper nontrivial subgroup $V_1$ of $\widehat{M}$ that is preserved by $G$.  The annihilator of $V_1$ in $M$,
\[
V^{\perp}_1 = \bigcap_{\chi \in V_1}\ker\chi,
\]
is then a proper nontrivial closed subgroup of $M$ that is normal in $G$, contradicting the minimality of $M$.
\end{proof}

\subsection{The groups $\Qb_p(\kappa)$}

The group $\Qb_p(\kappa)$ is the local direct product $\prod_X (\Qb_p,\Zb_p)$ where the indexing set $X$ has size $\kappa$.  The group $\Qb_p(\aleph_0)$ has been proposed (for example \cite{CT}) as a $p$-adic version of the infinite-dimensional separable Hilbert space, over which one can define operator algebras and representation theory analogous to the classical theory over $\Rb$ and $\Cb$.  However, in contrast to the other families, the author is not aware of a standard characterization of the groups $\Qb_p(\kappa)$ as abelian \lcsc groups.  We provide a new characterization here.

\begin{prop}\label{prop:Qpkappa_characterization}Let $A$ be an abelian \lcsc group and let $p \in \Pb$.  Then the following are equivalent:
\begin{enumerate}[(i)]
\item $A \cong \Qb_p(\kappa)$ for some $\kappa \in \Nb \cup \{\aleph_0\}$;
\item $A$ is totally disconnected and torsion-free, and both $P_p(A)$ and $pA$ are dense in $A$;
\item $A$ is totally disconnected and torsion-free, $P_p(A) = A$, and $\divi(A)$ is dense in $A$.
\end{enumerate}
\end{prop}

\begin{proof}
It is clear that (i) implies (ii), and we see that (ii) implies (iii) by Lemmas~\ref{lem:PG_closed} and~\ref{lem:pdense_divisible}.   It now remains to show that (iii) implies (i).  So from now on we assume that $A$ is totally disconnected and torsion-free, and that both $P_p(A)$ and $\divi(A)$ are dense in $A$.  In fact $P_p(A)=A$ by Lemma~\ref{lem:PG_closed}, so we can regard $A$ as a $\Zb_p$-module; in particular, since $A$ is torsion-free, it follows that $A$ is not discrete.

Fix a compact open subgroup $U$ of $A$.  Since $U$ is an infinite torsion-free second-countable pro-$p$ group, by Lemma~\ref{lem:Zp_direct} we have $U \cong \Zb^{\kappa}_p$ for some $\kappa \in \Nb \cup \{\aleph_0\}$.

Next, we note that $\divi(A) \cap U$ is a dense $\Zb_p$-submodule of $U$: given $\lambda \in \Zb_p$ and $a \in \divi(A)$, then $\lambda a$ has $n$-th root $\lambda(a/n)$ in $A$ for all $n \in \Nb$, so $\lambda a \in \divi(A) \cap U$.  In the case that $U \cong \Zb^n_p$ for some $n \in \Nb$, then all $\Zb_p$-submodules of $U$ are closed by Lemma~\ref{lem:Zp_direct}, so $\divi(A) \ge U$; in this case, $A$ is divisible, so $A = U^{\divi} \cong \Qb^n_p$.  Thus we may assume that $\kappa = \aleph_0$.  More explicitly, let us write $\Qb_p(\aleph_0)$ as $\Qb_p(\Nb) := \bigoplus_{\Nb}(\Qb_p,\Zb_p)$ and let $\delta_i$ be the element of $\bigoplus_{\Nb}(\Qb_p,\Zb_p)$ with $i$-th entry $1$ and all other entries $0$.  Note that $\{\delta_i \mid i \in \Nb\}$ is a basis for $\Zb^{\Nb}_p := \prod_{\Nb}\Zb_p$.

Note that, since $\Qb_p(\Nb)/\Zb^{\Nb}_p$ is torsion, we have $\Qb_p(\Nb) \le (\Zb^{\Nb}_p)^{\divi}$.  Similarly, $A/U$ is torsion, since $U$ is open and the $p^n$-th powers of every element of $A$ converge to the identity as $n \rightarrow \infty$; thus $A \le U^{\divi}$.  We aim to find an isomorphism $\theta: U \rightarrow \Zb^{\Nb}_p$ of topological groups whose divisible extension $\theta^{\divi}$ satisfies $\theta^{\divi}(A) = \Qb_p(\Nb)$.

Notice that we can write $A$ as a subgroup $M^{\divi} + U$ of $U^{\divi}$, where $M$ is a $\Zb_p$-submodule of $U$.  Indeed, since $A = \divi(A) + U$ and $U$ is essential, the submodule $M = \divi(A) \cap U$ will suffice.  In fact, it suffices to take a submodule $M$ that contains some multiple of an element of $a+U$ for each $a \in A$.  Since $A/U$ is countable we can take $M$ to be countably generated as a $\Zb_p$-module.  That is, we can take $M = \bigcup_{i \in \Nb}M_i$, where $(M_i)_{i \in \Nb}$ is an increasing sequence of finite-rank $\Zb_p$-submodules of $U$.

Given one such $\Zb_p$-submodule $M_i$, let $M'_i$ be the set of roots of $M_i$ in $U$.  Since $M_i$ has finite rank, it is closed by Lemma~\ref{lem:Zp_direct}; hence $M_i$ is isomorphic as a topological group to $\Zb^r_p$ for some $r \in \Nb$.  We then see that $M_i$ has countable index in $M^{\divi}_i \cong \Qb^r_p$.  Since $M_i$ is essential in $M'_i$, the inclusion of $M_i$ in $M'_i$ extends to a continuous surjection of $M^{\divi}_i$ onto $(M'_i)^{\divi}$.  It follows that in fact $M_i$ is open in $M'_i$, so $M'_i$ is closed in $U$.  Since $U$ is compact we conclude that $|M'_i:M_i|$ is finite.  We also see that $M'_i \le \divi(A)$ and that $M'_i$ is pure in $U$, so $M'_i$ is a direct summand of $U$ of the same rank as $M_i$.  Thus by replacing $M_i$ with $M'_i$, we may assume $M = \bigcup_{i \in \Nb}M_i$, where $(M_i)_{i \in \Nb}$ is an increasing sequence of direct summands of $U$.

By refining the sequence, we may also assume that $M_i$ has rank $i$; in other words, $M_1 = \cgrp{a_1}$ and thereafter $M_{i+1} = M_i \oplus \cgrp{a_{i+1}}$, for some sequence $(a_i)_{i \in \Nb}$ in $\divi(A) \cap U$.

Choose an ordered basis $(e_i)_{i \in \Nb}$ for $U$.

\emph{Claim: Subject to a suitable reordering of $(e_i)_{i \in \Nb}$: there are $\lambda_{ij} \in \Zb_p$ with $\lambda_{ij} = 0$ for $i > j$ and $\lambda_{ij} =1$ for $i=j$, and $a_i := \sum_{j \in \Nb}\lambda_{ij}e_j$, such that $\{a_1,\dots,a_n\}$ forms a basis for $M_n$ for all $n \in \Nb$.}

We proceed by induction starting from $n=0$, where $M_0 = \{0\}$; there is nothing to prove in the base case.  Let $n \ge 1$.  Given $k \in \Nb$, let $\pi_k$ be the projection of $U = \prod_{i \in \Nb}\cgrp{e_i}$ onto the direct summand $\prod_{i \le k}\cgrp{e_i}$ with respect to the basis $(e_i)_{i \in \N}$.  By the inductive hypothesis we have a suitable basis $\{a_1,\dots,a_{n-1}\}$ for $M_{n-1}$, and we see that
\[
\pi_{n-1}(M_{n-1}) = \prod_{i \le n-1}\cgrp{e_i} = \pi_{n-1}(M_n),
\]
so $M_n = M_{n-1} + \ker\pi_{n-1}$.  Considering ranks, it follows that in fact $M_n = M_{n-1} \oplus \cgrp{a_n}$ for some $a_n \in \ker\pi_{n-1}$.  Since $M_n$ is a direct summand of $U$, we see that $a_n \not\in pU$, and hence if we write $a_n$ as $\prod_{j \ge n}\lambda_je_j$, there is some $k \ge n$ such that $\lambda_k \in \Zb^*_p$.  By replacing $a_n$ with $\lambda^{-1}_ka_n$ we may assume $\lambda_k = 1$, and then by swapping $e_k$ and $e_n$ we may assume that $k = n$.  (Note that this last step does not change $e_1,\dots,e_{n-1}$, so as we run the argument, we will eventually arrive at a fixed ordered basis for $U$ with the required properties.)  We now see that $\{a_1,\dots,a_n\}$ has the required form, completing the proof of the claim.

We now consider the set $X = \{a_i \mid i \in \Nb\}$ provided to us by the claim.  Note that the restriction on the coefficients of $a_i$ with respect to the basis $\{e_i \mid i \in \Nb\}$ ensures that $a_i \rightarrow 0$ as $i \rightarrow \infty$.  The fact that $M$ is dense in $U$ ensures that $X$ is a topological generating set for $U$.  Moreover, $X$ is topologically linearly independent over $\Zb_p$: given an equation of the form $\sum_{i \in \Nb}\mu_ia_i = 0$ with $\mu_i \in \Zb_p$, we see by applying the maps $\pi_n$ that $\mu_n=0$ for all $n \in \Nb$.  Thus $X$ is a basis for $U$.

We now carry out the same argument for writing $\Qb_p(\Nb)$ as a subgroup $N^{\divi} + \Zb^{\Nb}_p$ of $(\Zb^{\Nb}_p)^{\divi}$.  In this case, we see that it suffices to take $N = \bigcup_{i \in \Nb}N_i$, where $N_i = \cgrp{\delta_1,\dots,\delta_i}$.

Since $(a_i)_{i \in \Nb}$ and $(\delta_i)_{i \in \Nb}$ are bases of $U$ and $\Zb^{\Nb}_p$ respectively, there is a unique isomorphism $\theta: U \rightarrow \Zb^{\Nb}_p$ of topological groups such that $\theta(a_i) = \delta_i$; this extends to an isomorphism $\theta^{\divi}: U^{\divi} \rightarrow (\Zb^{\Nb}_p)^{\divi}$.  It is then clear that $\theta^{\divi}(M) = N$ and hence $\theta^{\divi}(A) = \Qb_p(\Nb)$, so $\theta^{\divi}$ restricts to an isomorphism from $A$ to $\Qb_p(\Nb)$.  This completes the proof that (iii) implies (i) and hence that all three statements are equivalent.
\end{proof}

If $A \cong \Qb_p(\kappa)$, then the cardinal $\kappa$ is the rank in the sense of Lemma~\ref{lem:Zp_direct} of any compact open subgroup of $A$.  Thus $\kappa$ is uniquely determined by $A$.  Clearly also $p$ is uniquely determined.

\begin{example}\label{ex:Qpkappa}Fix $p \in \Pb$ and $\kappa \in \Nb \cup \infty$.  We have a compactly generated metabelian group $G = \Qb_p \rtimes \grp{s}$, where $s$ acts on $\Qb_p$ by multiplication by $p$.  Note that the proper nontrivial closed subgroups of $\Qb_p$ are exactly its compact open subgroups, which are all conjugate in $G$; from here it is easy to see that $\Mon(G) = \Qb_p$.

Generalizing this example, we realize $\Qb_p(\kappa)$ as the monolith of
\[
\bigoplus_{C_{\kappa}} (\Qb_p \rtimes \grp{s},\Zb_p) \rtimes C_{\kappa},
\]
where $C_{\kappa}$ acts by shifting the indexing set of the local direct product.
\end{example}

\section{The classification of topologically characteristically simple abelian \lcsc groups}\label{sec:classification}

We now show that the five families described in the previous section comprise the whole of the class $\ms{MA}$ of topologically characteristically simple abelian \lcsc groups.  With the characterizations of the previous section in hand, the proof is straightforward: we show that any $A \in \ms{MA}$ must have one of the sets of properties characterizing one of the families.

\begin{lem}\label{lem:characteristic_dichotomy}
Let $A \in \ms{MA}$.  Then the following holds:
\begin{enumerate}[(1)]
\item Either $A$ has prime exponent, or else $A$ is torsion-free.
\item If $A$ is torsion-free then $pA$ is dense for all $p \in \Pb$.
\item $A$ is either connected or totally disconnected.
\item $P(A)$ is trivial or equal to $A$.
\item If $A$ is totally disconnected and $P(A) = A$, then $P_p(A) = A$ for some $p \in \Pb$.
\end{enumerate}
\end{lem}

\begin{proof}
(1) We note that for all $p \in \Pb$, the subgroup $\Omega_p(A) := \{a \in A \mid pa = 0\}$ is closed, and hence belongs to $\{\{0\},A\}$.  If $\Omega_p(A) = A$ for some $p \in \Pb$, then $A$ has exponent $p$.  Otherwise, $\Omega_p(A) = \{0\}$ for all $p \in \Pb$, from which it follows for all $n \in \Nb$ (by induction on the number of prime factors of $n$) that the only solution in $A$ to the equation $na = 0$ is $a = 0$, that is, $A$ is torsion-free.

(2) Suppose $A$ is torsion-free and let $p \in \Pb$.  Then $pA$ is an injective image of $A$, hence nontrivial; thus $pA$ is dense in $A$.

(3) The connected component $A^\circ$ of the identity in $A$ is a closed characteristic subgroup, so either $A^\circ = A$ or $A^\circ = \{0\}$.

(4) This is clear, as $P(A)$ is a closed characteristic subgroup.

(5) If $A$ has exponent $p \in \Pb$ then $P_p(A) = A$, so by (1) we may assume $A$ is torsion-free.  In particular, $P(A)$ is nontrivial and torsion-free, so $A$ is not discrete.  It follows that $A$ has an infinite profinite open subgroup $U$.  Since $U$ is nontrivial, it must have a nontrivial $p$-Sylow subgroup $S$ for some $p \in \Pb$.  We then see that $P_p(A) \cap U = S$; in particular, $P_p(A)$ is nontrivial and closed.  Since $P_p(A)$ is clearly also characteristic, we deduce that $A = P_p(A)$.
\end{proof}

\begin{thm}\label{thm:classification}
Let $A$ be a nontrivial abelian \lcsc group.  Then the following are equivalent:
\begin{enumerate}[(i)]
\item $A$ is isomorphic to one of the following, where $n \in \Nb$, $\kappa \in \Nb \cup \{\aleph_0\}$ and $p \in \Pb$ as applicable:
\[
C^n_p; \; C^{\aleph_0}_p; \; C^{(\aleph_0)}_p; C^{\aleph_0}_p \oplus C^{(\aleph_0)}_p;  \; \Rb^n; \; \Qb^{(\aleph_0)}; \; \widehat{\Qb}^{\aleph_0}; \; \Qb_p(\kappa).
\]
\item $A$ is isomorphic to the monolith of some soluble \lcsc group of derived length at most $3$.
\item $A$ is topologically characteristically simple.
\item $A$ satisfies the statements (1)--(5) of Lemma~\ref{lem:characteristic_dichotomy}.
\end{enumerate}
\end{thm}

\begin{proof}
The examples given in the previous section show that (i) implies (ii).  Clearly (ii) implies (iii), and (iii) implies (iv) by Lemma~\ref{lem:characteristic_dichotomy}.  It remains to show that (iv) implies (i).

If $A$ has prime exponent then $A \cong C^{\kappa}_p \oplus C^{(\kappa')}_p$ for some $\kappa \in \Nb \cup \{\aleph_0\}$ and $\kappa' \in \{0,\aleph_0\}$ by Proposition~\ref{prop:elab_characterization}, so we may assume that $A$ does not have prime exponent; thus $A$ is torsion-free by (1), and hence by (2), $pA$ is dense in $A$ for all $p \in \Pb$.

By (3), $A$ is either connected or totally disconnected, and by (4), $P(A)$ is either trivial or dense.  Consider the four cases in turn.

\emph{$A$ is connected and $P(A) = \{0\}$.}  By Proposition~\ref{prop:R_characterization}, we have $A \cong \Rb^n$ for some $n \in \Nb$.

\emph{$A$ is connected and $P(A)$ is dense.}  Since $A$ is also torsion-free, we conclude by Proposition~\ref{prop:Qhat_characterization} that $A \cong \widehat{\Qb}^\kappa$ for some $\kappa \in \Nb \cup \{\aleph_0\}$.

\emph{$A$ is totally disconnected and $P(A) = \{0\}$.}  Since $pA$ is dense in $A$ for all $p \in \Pb$, we conclude by Proposition~\ref{prop:Q_characterization} that $A \cong \Qb^{(\kappa)}$ for some $\kappa \in \Nb \cup \{\aleph_0\}$.

\emph{$A$ is totally disconnected and $P(A)$ is dense.}  In this case $A = P_p(A)$ for some $p \in \Pb$ by (5).  Since also $A$ is torsion-free and $pA$ is dense, we conclude by Proposition~\ref{prop:Qpkappa_characterization} that $A \cong \Qb_p(\kappa)$ for some $\kappa \in \Nb \cup \{\aleph_0\}$.

In all cases (i) holds, completing the cycle of implications.
\end{proof}

We note the following consequence of Theorem~\ref{thm:classification} together with the constructions of the previous section.

\begin{cor}
With the exception of the groups $\Qb^n$ and $\widehat{\Qb}^n$ for $n \in \Nb$, every abelian topologically characteristically simple \lcsc group occurs as the monolith of a compactly generated soluble \lcsc group of derived length at most $3$.
\end{cor}

Theorem~\ref{mainthm} is now clear.

\begin{example}\label{ex:minimal}
Let $p \in \Pb$ and let $\Kb$ be the additive group of one of the fields $\Fb_p((t))$, $\Rb$ or $\Qb_p$.  If $\Kb$ is $\Fb_p((t))$ or $\Qb_p$, let $A_{\Kb} = \grp{s}$ where $s$ acts on $\Fb_p((t))$ or $\Qb_p$ as multiplication by $t$ or $p$ respectively, and if $\Kb = \Rb$ we set $A_{\Kb} = \Rb^*_{>0}$.  Thus $\Kb \rtimes A_{\Kb}$ is a compactly generated metabelian \lcsc group with $\Kb$ as its monolith.  We then form a group $B_{\Kb} = (\Kb \times \Kb) \rtimes A_{\Kb}$, where $A_{\Kb}$ has the same action on each copy of $\Kb$.  This time, we see that $B_{\Kb}$ has $2^{\aleph_0}$ minimal closed normal subgroups isomorphic to $\Kb$: namely, for each $\lambda \in \Kb$ there is a minimal closed normal subgroup
\[
L_{\Kb,\lambda} := \{(a,\lambda a) \mid a \in \Kb\}.
\]
In particular, this shows that $\Fb_p((t))$, $\Rb$ and $\Qb_p$ can occur $2^{\aleph_0}$ times as a minimal closed normal subgroup of a compactly generated \lcsc group.  Generalizing this example, we see for $n \in \Nb$ that $\Kb^n$ occurs $2^{\aleph_0}$ times as a minimal closed normal subgroup of $B_{\Kb} \wr C_n$, by taking a subgroup of the base group of the form $(L_{\Kb,\lambda})^n$.  In the case $\Kb = \Qb_p$, we can form the group
\[
\bigoplus_{\Zb} (B_{\Qb_p},\Zb_p \times \Zb_p) \rtimes \Zb,
\]
where $\Zb$ acts by shifting the indexing set; this time $\Qb_p(\aleph_0)$ occurs $2^{\aleph_0}$ times as a minimal closed normal subgroup.
\end{example}

\bibliographystyle{amsplain}
\bibliography{biblio}{}

\end{document}